\documentclass{article}

\usepackage{
amsthm, amssymb}

\newtheorem{theorem}{Theorem}[section]
\newtheorem{lemma}[theorem]{Lemma}
\newtheorem{proposition}[theorem]{Proposition}
\newtheorem{corollary}[theorem]{Corollary}

\newtheorem{definition}[theorem]{Definition}

\newtheorem{example}[theorem]{Example}

\title{Twisted conjugacy in braid groups}

\author{ {\bf Juan Gonz\'alez-Meneses} \\ $ $\\ {\em Dep. \'Algebra, Universidad de Sevilla, Apdo. 1160} \\ {\em 41080
Sevilla, Spain} \\ {\em e-mail:} meneses@us.es \\ $ $ \\
{\bf Enric Ventura}\\ $ $\\ {\em Dept. Mat. Apl. III, Universitat Polit\`ecnica de Catalunya,}\\ {\em Manresa,
Barcelona, Catalunya} \\ {\em e-mail:} enric.ventura@upc.edu}

\date{\today}

\begin{document}

\maketitle

\begin{abstract}
In this note we solve the twisted conjugacy problem for braid groups, i.e. we propose an algorithm which, given two
braids $u,v\in B_n$ and an automorphism $\varphi \in Aut (B_n)$, decides whether $v=(\varphi (x))^{-1}ux$ for some
$x\in B_n$. As a corollary, we deduce that each group of the form $B_n \rtimes H$, a semidirect product of the braid
group $B_n$ by a torsion-free hyperbolic group $H$, has solvable conjugacy problem.
\end{abstract}


\section{Introduction}

Let $G$ be a group, and $\varphi \in Aut(G)$ an automorphism (which we shall write on the left of the argument,
$g\mapsto \varphi(g)$). We say that two elements $u, v\in G$ are \emph{$\varphi$-twisted conjugated}, denoted
$u\sim_{\varphi} v$, if there exists $x\in G$ such that $v=(\varphi (x))^{-1}ux$. It is straightforward to see that
$\sim_{\varphi}$ is an equivalence relation on $G$, which coincides with standard conjugation in the case $\varphi =Id$
(we shall use the symbol $\sim$ instead of $\sim_{Id}$). Reidemeister was the first author considering this concept
(see~\cite{Reid}), which has an important role in modern Nielsen fixed point theory.

As one might expect, in general, twisted conjugacy classes are much more complicated to understand than standard
conjugacy classes in a group $G$. For instance, algorithmic recognition of them already presents big differences. The
\emph{twisted conjugacy problem} for a group $G$ consists on finding an algorithm which, given an automorphism $\varphi
\in Aut(G)$ and two elements $u,v\in G$, decides whether $u\sim_{\varphi} v$ or not. While the conjugacy problem (i.e.
the $Id$-twisted conjugacy problem) is very easy for free groups, both conceptually and computationally, the twisted
conjugacy problem is solvable but much harder in both senses, see Theorem~1.5 in~\cite{BMMV}.

Of course, a positive solution to the twisted conjugacy problem automatically gives a solution to the (standard)
conjugacy problem, which in turn provides a solution to the word problem. The existence of a finitely presented group
$G$ with solvable word problem but unsolvable conjugacy problem is well known (see~\cite{M1}). In the same direction,
there exists a finitely presented group with solvable conjugacy problem, but unsolvable twisted conjugacy problem (see
Corollary~4.9 in~\cite{BMV}).

A subgroup $A\leqslant Aut(G)$ is said to be \emph{orbit decidable} if there is an algorithm which, given two elements
$u,v\in G$ as input, decides whether one can be mapped to the other up to conjugacy, by some automorphism in $A$, i.e.
whether $v\sim \alpha(u)$ for some $\alpha \in A$ (see~\cite{BMV} for more details). For example, the conjugacy problem
in $G$ coincides precisely with the orbit decidability of the trivial subgroup $\{Id\}\leqslant Aut(G)$.

Let
 $$
1\longrightarrow F\stackrel{\alpha}{\longrightarrow} G\stackrel{\beta}{\longrightarrow} H\longrightarrow 1.
 $$
be a short exact sequence of groups. Since $\alpha(F)$ is normal in $G$, for every $g\in G$, the right conjugation
$\gamma_g$ of $G$ induces an automorphism of $F$, $x\mapsto g^{-1}xg$, which will be denoted $\varphi_g \in Aut(F)$
(note that, in general, $\varphi_g$ does not belong to $Inn(F)$). It is clear that the set of all such automorphisms,
 $$
A_{G} =\{ \varphi_g \mid g\in G\},
 $$
forms a subgroup of $Aut(F)$ containing $Inn(F)$. We shall refer to it as the \emph{action subgroup} of the given short
exact sequence.

Such a sequence is said to be \emph{algorithmic} provided it is given along with algorithms: (i) to compute in the
groups $F$, $G$ and $H$ (i.e. multiply and invert elements), and compute images under $\alpha$ and $\beta$; (ii) to
compute one pre-image in $G$ of any given element in $H$; and (iii) to compute pre-images in $F$ of elements in $G$
mapping to the trivial element in $H$. The typical example (though not the unique one) of an algorithmic short exact
sequence occurs when groups are given by finite presentations and maps are given by images of generators. In fact, (i)
is immediate, we can use the positive part of the membership problem for $\beta(G)$ in $H$ to compute pre-images in $G$
of elements in $H$, and use the positive part of the membership problem for $\alpha(F)$ in $G$ to compute pre-images in
$F$ of elements in $G$ mapping to $1_H$ (see Section~2 in~\cite{BMV}).

Assuming certain conditions on the groups $F$ and $H$, the main result in~\cite{BMV} establishes the following
characterization of the solvability of the conjugacy problem for $G$, in terms of the orbit decidability for the
corresponding action subgroup.

\begin{theorem}[Bogopolski, Martino, Ventura~\cite{BMV}]\label{ses}
Let
 $$
1\longrightarrow F\stackrel{\alpha}{\longrightarrow} G\stackrel{\beta}{\longrightarrow} H\longrightarrow 1.
 $$
be an algorithmic short exact sequence of groups such that
\begin{itemize}
\item[(i)] $F$ has solvable twisted conjugacy problem,
\item[(ii)] $H$ has solvable conjugacy problem, and
\item[(iii)] for every $1\neq h\in H$, the subgroup $\langle h\rangle$ has finite index in its centralizer $C_H(h)$, and
there is an algorithm which computes a finite set of coset representatives, $z_{h,1},\ldots ,z_{h,t_h}\in H$,
 $$
C_H(h)=\langle h \rangle z_{h,1}\sqcup \cdots \sqcup \langle h \rangle z_{h,t_h}.
 $$
\end{itemize}
Then, the conjugacy problem for $G$ is solvable if and only if the action subgroup $A_G \leqslant Aut(F)$ is orbit
decidable.
\end{theorem}

Many groups satisfy conditions (ii) and (iii); for example, they are easily verified for a finitely generated free
group, and with a bit more work, they can also be proven for torsion-free hyperbolic groups, see Proposition~4.11
in~\cite{BMV}.

On the other hand, solvability of the twisted conjugacy problem is a stronger condition on $F$. In this sense, the
introduction of~\cite{BMV} contains the following comment: ``\emph{In light of Theorem~\ref{ses}, it becomes
interesting, first, to collect groups $F$ where the twisted conjugacy problem can be solved. And then, for every such
group $F$, to study the property of orbit decidability for subgroups of $Aut(F)$: every orbit decidable (undecidable)
subgroup of $Aut(F)$ will correspond to extensions of $F$ having solvable (unsolvable) conjugacy problem}".

The goal of the present paper is to contribute a new result into this direction, taking as a base group the braid
group, $F=B_n$.

Consider the braid group on $n$ strands, given by the classical presentation
\begin{equation}\label{E:presentation}
B_n =\left\langle \sigma_1,\sigma_2,\ldots ,\sigma_{n-1} \left\vert \begin{array}{ll} \sigma_i \sigma_j =\sigma_j
\sigma_i & |i-j|> 1 \\ \sigma_i \sigma_{j} \sigma_i =\sigma_{j} \sigma_i \sigma_{j} & |i-j|=1
\end{array} \right. \right\rangle
\end{equation}

It is well known that the conjugacy problem is solvable in $B_n$. The first, non-efficient solution was given by
Garside~\cite{Garside}. It was subsequently improved in~\cite{EM,Franco_GM,Birman_Ko_Lee,Gebhardt,Gebhardt_GM}, in such
a way that the current solution is very efficient in most cases.

\begin{theorem}[Garside~\cite{Garside}]\label{cpbn}
The conjugacy problem is solvable in $B_n$.
\end{theorem}

Also, the automorphism group of $B_n$ is quite well understood. Among other results, the following one will be crucial
for our argumentation.

\begin{theorem}[Dyer, Grossman~\cite{DG}]\label{out}
Let $B_n$ be the braid group on $n$ strands. Then $|Out(B_n)|=2$. More precisely, $Aut(B_n)=Inn(B_n)\sqcup
Inn(B_n)\cdot \varepsilon$, where $\varepsilon \colon B_n \to B_n$ is the automorphism which inverts each generator,
$\sigma_i \mapsto \sigma_i^{-1}$.
\end{theorem}

Using the above two results, we will solve the twisted conjugacy problem in $B_n$, and the orbit decidability problem
for every subgroup $A\leqslant Aut(B_n)$. As a consequence, we deduce that the conjugacy problem is solvable in certain
extensions of $B_n$.

\textbf{Theorem~\ref{tcp}.} \emph{The twisted conjugacy problem is solvable in the braid group $B_n$.}

\textbf{Theorem~\ref{od}.} \emph{Every finitely generated subgroup $A\leqslant Aut(B_n)$ is orbit decidable.}

\textbf{Theorem~\ref{cp}.} \emph{Let $G=B_n \rtimes H$ be an extension of the braid group $B_n$ by a finitely generated
group $H$ satisfying conditions (ii) and (iii) above (for instance, take $H$ torsion-free hyperbolic). Then, $G$ has
solvable conjugacy problem.}

The structure of the paper is as follows. In Section~2 we review some known facts about normal forms for braids that
will be used later. In Section~3 we determine a well defined finite subset of each $\varepsilon$-twisted conjugacy
class in $B_n$. And in Section~4 we give an algorithm to construct such set from a given element in the class, solving
the twisted conjugacy problem in $B_n$. Finally, in Section~5 we solve the orbit decidability problem for subgroups of
$Aut(B_n)$ and conclude Theorem~\ref{cp}.

\section{Normal forms of braids}

In this section we will recall the notion of normal form for braids, as explained in~\cite[Chapter 9]{Thurston} and
\cite{EM}, and we shall also provide some technical lemmas that will be used to prove our main results.

In the braid group $B_n$, an element is called {\bf positive} if it can be written as a product of non-negative powers
of the generators $\sigma_1,\ldots,\sigma_{n-1}$. It turns out that if we regard the standard presentation of the braid
group~(\ref{E:presentation}) as a monoid presentation, it yields a monoid $B_n^+$ which embeds in $B_n$, and is
precisely the submonoid of positive braids~\cite{Garside}. This means that two positive words represent the same braid
if and only if one can be obtained from the other by a finite sequence of the following operations: Either replacing a
subword $\sigma_i\sigma_j$ by $\sigma_j\sigma_i$ for $|i-j|>1$, or replacing a subword $\sigma_i\sigma_j\sigma_i$ by
$\sigma_j\sigma_i\sigma_j$ for $|i-j|=1$.

There is a partial order $\preccurlyeq$ on the elements of $B_n$, called the {\bf prefix order}, defined by
$a\preccurlyeq b$ if and only if $a^{-1}b$ is positive. If $a$ and $b$ are positive this means that $b$ can be written
as a positive word in which $a$ appears as a prefix. There is also a {\bf suffix order}, $\succcurlyeq$, defined by
$a\succcurlyeq b$ if and only if $ab^{-1}$ is positive.  These orders are known to be lattice orders, meaning that for
every $a,b\in B_n$ there is a unique greatest common divisor $a\wedge b$ (resp. $a\wedge_R b$) and a unique least
common multiple $a\vee b$ (resp. $a\vee_R b$) with respect to $\preccurlyeq$ (resp. $\succcurlyeq$).

The order $\preccurlyeq$ is, by definition, invariant under left-multiplication. That is, $a\preccurlyeq b
\Leftrightarrow ca\preccurlyeq cb$ for all $a,b,c\in B_n$. This implies that $cx\wedge cy = c(x\wedge y)$ and $cx\vee
cy = c(x\vee y)$ for all $c,x,y\in B_n$. Similarly, $\succcurlyeq $ is invariant under right-multiplication, and one
has $xc\wedge_R yc = (x\wedge_R y)c$ and $xc\vee_R yc = (x\vee_R y)c$ for all $c,x,y\in B_n$.

The braid group $B_n$ has a special element called {\bf Garside element} or \mbox{\bf half twist},
 $$
\Delta=\sigma_1(\sigma_2\sigma_1)(\sigma_3\sigma_2\sigma_1) \cdots (\sigma_{n-1}\cdots \sigma_1).
 $$
Conjugation by $\Delta$ preserves $\preccurlyeq$ and $\succcurlyeq$. We denote by $\tau$ the inner automorphism of
$B_n$ defined by $\Delta$, that is, $\tau(x)=\Delta^{-1}x\Delta$ for all $x\in B_n$. We recall that the center of $B_n$
is infinite cyclic, generated by $\Delta^2$. Hence $\tau$ preserves $\preccurlyeq$ and $\succcurlyeq$ (thus it
preserves $\wedge$, $\vee$, $\wedge_R$ and $\vee_R$), and $\tau^2=\mbox{id}$.

The set of positive prefixes of $\Delta$, denoted $[1,\Delta]=\{s\in B_n; \; 1\preccurlyeq s\preccurlyeq \Delta\}$, is
called the set of {\bf simple elements} of $B_n$. This set is finite, namely it has $n!$ elements. Simple elements are
the building blocks that conform the usual normal forms of braids. A simple element will be said to be {\bf proper} if
it is  neither $1$ nor $\Delta$.

The {\bf right complement} $\partial(s)$ of a simple element $s$ is a simple element $t$ such that $st=\Delta$, that
is, $\partial(s)=s^{-1}\Delta$. The map $\partial$ is a bijection of the set of simple elements. Moreover,
$\partial^2=\tau$. The {\bf left complement} of a simple element is precisely $\partial^{-1}(s)=\Delta s^{-1}$. If a
positive element is written as a product of two simple elements $s_1s_2$, we say that such a decomposition is {\bf left
weighted} if $s_1$ is the maximal simple prefix of $s_1s_2$, that is $s_1s_2\wedge \Delta = s_1$, or alternatively
(multiplying from the left by $s_1^{-1}$), if $s_2 \wedge \partial(s_1)=1$.   We say that the decomposition $s_1s_2$ is
{\bf right weighted} if $s_2$ is the maximal simple suffix of $s_1s_2$, that is $s_1s_2\wedge_R \Delta = s_2$, or
alternatively (multiplying from the right by $s_2^{-1}$), if $s_1 \wedge_R \partial^{-1}(s_2)=1$.

Given an element $x\in B_n$, we say that a decomposition $x=\Delta^p x_1\cdots x_r$ is the {\bf left normal form} of
$x$ if $p$ is the maximal integer such that $\Delta^{-p}x$ is positive, each $x_i$ is a proper simple element, and
$x_ix_{i+1}$ is left weighted for $i=1,\ldots,r-1$.  We say that a decomposition $x=x_1'\cdots x_r' \Delta^p$ is the
{\bf right normal form} of $x$ if $p$ is the maximal integer such that $x\Delta^{-p}$ is positive, each $x_i'$ is a
proper simple element, and $x_i'x_{i+1}'$ is right weighted for $i=1,\ldots,r-1$.  The left and right normal forms are
unique decompositions, and the numbers $p$ and $r$ are determined by $x$ and do not depend on the normal form (left or
right) which is used to define them. In this way, one defines the {\bf infimum}, {\bf supremum} and {\bf canonical
length} of $x$ as, respectively, $\inf(x)=p$, $\sup(x)=p+r$ and $\ell(x)=r$.

It will be convenient for our purposes to use the following notation. When we deal with a positive element $x$, and we
say that its left normal form is $x=x_1\cdots x_r$, (with no power of $\Delta$ on the left), we are allowing some of
the initial factors to be equal to $\Delta$. That is, if $\inf(x)=p>0$, this will mean that $x_1=\cdots =x_p=\Delta$,
so the actual normal form of $x$ would be $\Delta^p x_{p+1}\cdots x_r$.

There is still another normal form that we shall use. It is well known~\cite{Thurston,Charney} that, for every $x\in
B_n$ there exist unique positive elements $u$ and $v$, with $u\wedge v=1$, such that $x=u^{-1}v$. If the left normal
forms of $u$ and $v$ are, respectively, $u=u_1\cdots u_r$ and $v=v_1\cdots v_s$, the {\bf mixed normal form} of $x$ is
defined to be $x=u_r^{-1}\cdots u_1^{-1} v_1\cdots v_s$.  We recall from~\cite{Thurston} that, if $x$ can be written as
$x=u^{-1}v$ with $u$ and $v$ positive elements with left normal forms $u_1\cdots u_r$ and $v_1\cdots v_s$, then
$u_r^{-1}\cdots u_1^{-1}v_1\cdots v_s$ is the mixed normal form of $x$ if and only if $u_1\wedge v_1=1$.

We remark~\cite{Thurston} that if $x=u_r^{-1}\cdots u_1^{-1}v_1\cdots v_s $ is in mixed normal form as above, the left
normal form of $u^{-1}$ is $\Delta^{-r}u_r'\cdots u_1'$ (where $u_i'=\partial^{-2i-1}(u_i)$), and the left normal form
of $x$ is equal to $x=\Delta^{-r} u_r'\cdots u_1' v_1\cdots v_s$. Therefore, from the mixed normal form one can already
obtain $\inf(x)=-r$, $\ell(x)=r+s$ and hence $\sup(x)=s$.

The following technical results will be used later.

\begin{lemma}\label{L:cancelling}
Let $a$ and $b$ be positive braids whose left normal forms are $a= a_1\cdots a_r$, $b=b_1\cdots b_s$, and whose right
normal forms are $a=a_1'\cdots a_r'$, $b=b_1'\cdots b_s'$. Consider $x=a^{-1}b$. If $\ell(x)\leq r+s-2k+1$ for some
integer $k > 0$, then either $a_1'\cdots a_{k}'\preccurlyeq b_1\cdots b_{k}$ or $b_1'\cdots b_{k}'\preccurlyeq
a_1\cdots a_{k}$.
\end{lemma}

\begin{proof}
Let $d=a\wedge b$, and write $a=d\alpha$ and $b=d\beta$. Then $x=a^{-1}b=\alpha^{-1}\beta$, where $\alpha$ and $\beta$
are positive elements such that $\alpha\wedge \beta=1$. Hence $\sup(\alpha)+\sup(\beta)+2k-1 = \ell(x)+2k-1 \leq
r+s=\sup(a)+\sup(b)$. This implies that either $\sup(\alpha)+k\leq \sup(a)$ or $\sup(\beta)+k\leq \sup(b)$.

Suppose that $\sup(\alpha)+k\leq \sup(a)=r$. This means that $\alpha$ can be written as a product of at most $r-k$
simple elements. But $d\alpha = a = a_1'\cdots a_r'$, where the latter decomposition is in right normal form. It
follows that $\alpha$ must be a suffix of $a_{k+1}'\cdots a_r'$, and then $a_1'\cdots a_{k}'\preccurlyeq d$. Hence
$a_1'\cdots a_{k}'\preccurlyeq d\beta=b= b_1\cdots b_s$. Since the latter decomposition is in left normal form, one
finally obtains $a_1'\cdots a_{k}'\preccurlyeq b_1\cdots b_{k}$. In the case $\sup(\beta)+k \leq \sup(b)=s$, one can
apply the above reasoning to $\beta$ and $b$, to obtain $b_1'\cdots b_{k}'\preccurlyeq a_1\cdots a_{k}$.
\end{proof}

Let us denote by $\varepsilon$ the automorphism of $B_n$ that sends $\sigma_i$ to $\sigma_i^{-1}$ for $i=1,\ldots,n-1$.
Also, let $\mbox{rev}: \: B_n \rightarrow B_n$ be the anti-automorphism that sends each $\sigma_i$ to itself, that is,
it sends a braid represented by a word $w$, to the braid represented by the same word written backwards. We will write,
for every $x\in B_n$, $\mbox{rev}(x)=\overleftarrow{x}$.  Let us also denote $\mbox{inv}:\: B_n \rightarrow B_n$ the
anti-automorphism $\mbox{inv}(x)=x^{-1}$.  Notice that the composition of any two of the maps in $\{\varepsilon,
\mbox{rev}, \mbox{inv}\}$, in any order, yields the third one.

\begin{lemma}\label{L:sup_and_palindromic}
Let $x$ be a positive braid with $\sup(x)=r+k$, where $r\geq k\geq 1$. If $\ell(\varepsilon(x)x)\leq 2r+1$ then there
exist positive braids $a$ and $b$ such that $x=\overleftarrow{a}ba$ and $\sup(b)\leq r$.
\end{lemma}

\begin{proof}
Let $x_1\cdots x_{r+k}$ be the left normal form of $x$ and let $y_1\cdots y_{r+k}$ be its right normal form. Hence
$\overleftarrow{x_{r+k}} \cdots \overleftarrow{x_1}$ is the right normal form of $\overleftarrow{x}$ and
$\overleftarrow{y_{r+k}}\cdots \overleftarrow{y_1}$ is its left normal form. Notice that $\varepsilon(x)x=
(\overleftarrow{x})^{-1}x$. Hence,  if $\ell(\varepsilon(x)x)=\ell((\overleftarrow{x})^{-1}x)\leq 2r+1$,
Lemma~\ref{L:cancelling} tells us that either $\overleftarrow{x_{r+k}}\cdots \overleftarrow{x_{r+1}}\preccurlyeq
x_1\cdots x_k$ or $y_1\cdots y_k\preccurlyeq \overleftarrow{y_{r+k}}\cdots \overleftarrow{y_{r+1}}$.

Suppose that $\overleftarrow{x_{r+k}}\cdots \overleftarrow{x_{r+1}}\preccurlyeq x_1\cdots x_k$, and write $x_1\cdots
x_k=\overleftarrow{x_{r+k}}\cdots \overleftarrow{x_{r+1}}\alpha$ for some positive $\alpha$. Since $r\geq k$, one can
then write
$$
x=\left(\overleftarrow{x_{r+k}}\cdots \overleftarrow{x_{r+1}}\right)\alpha x_{k+1}\cdots x_r (x_{r+1}\cdots x_{r+k}),
$$
so the result follows in this case taking $a=x_{r+1}\cdots x_{r+k}$ and $b=\alpha x_{k+1}\cdots x_{r}$ (if $k=r$ then
$b=\alpha$).  Notice that $\sup(b)\leq r$ as $b$ is a suffix of $x_1\cdots x_r$.

Now suppose that $y_1\cdots y_k\preccurlyeq \overleftarrow{y_{r+k}}\cdots \overleftarrow{y_{r+1}}$, and write
$\overleftarrow{y_{r+k}}\cdots \overleftarrow{y_{r+1}} = y_1\cdots y_k\beta$ for some positive $\beta$, which is
equivalent to $y_{r+1}\cdots y_{r+k}=\overleftarrow{\beta} \overleftarrow{y_k}\cdots \overleftarrow{y_1}$. Then $x =
(y_1\cdots y_k) y_{k+1}\cdots y_{r} \overleftarrow{\beta} \left( \overleftarrow{y_k}\cdots \overleftarrow{y_1}
\right)$. Taking $a=\overleftarrow{y_k}\cdots \overleftarrow{y_1}$ and $b=y_{k+1}\cdots y_{r} \overleftarrow{\beta}$,
which is a prefix of $y_{k+1}\cdots y_{k+r}$, the result follows also in this case.
\end{proof}

We define a {\it palindromic-free} braid as a positive braid $x$ that cannot be decomposed as $x=\overleftarrow{a}ba$
for positive braids $a$ and $b$, where $a$ is nontrivial (see the equivalent definition~\ref{D:palindromic_free}).
Palindromic-free braids will be crucial to show our main results. The above Lemma implies the following.

\begin{corollary}\label{C:pal-free_double_length}
Let $u$ be a positive braid with $\ell(x)=m$. Then $\ell(\varepsilon(x)x)\leq 2m$. If moreover $u$ is palindromic-free
and $m>1$, then $\ell(\varepsilon(x)x)=2m$.
\end{corollary}

\begin{proof}
Recall that $\varepsilon(x)= (\overleftarrow{x})^{-1}$. Since the canonical length of a braid is preserved under
reversing (by symmetry of the relations in $B_n$) and also under taking inverses (by~\cite{EM}), it follows that
$\ell(\varepsilon(x))=m$. Multiplying two braids of canonical length $m$ yields a braid of canonical length at most
$2m$, hence $\ell(\varepsilon(x)x)\leq 2m$.

If $u$ is palindromic-free and $m>1$ we have the equality, as if we had $\ell(\varepsilon(x)x)< 2m$, then by setting
$r=m-1$ and $k=1$, we would have $\ell(\varepsilon(x)x)\leq 2r+1$, which by Lemma~\ref{L:sup_and_palindromic} implies
that $x$ is not palindromic-free.
\end{proof}

\section{$\varepsilon$-twisted conjugacy and palindromic-free braids}

Due to Theorem~\ref{out}, the twisted conjugacy problem in $B_n$ will easily reduce to the $\varepsilon$-twisted
conjugacy problem, namely given two braids $u,v\in B_n$ decide whether there exists another one $w\in B_n$ such that
 $$
v=(\varepsilon (w))^{-1}uw.
 $$
This problem has a very particular nature because $(\varepsilon (w))^{-1}=\overleftarrow{w}$, i.e.
$\varepsilon$-twisted conjugating $u$ by $w$ amounts to multiply $u$ on the right by $w$ and on the left by
$\overleftarrow{w}$, $v=\overleftarrow{w}uw$. Let us concentrate on this case, where the twisting is given by
$\varepsilon$.

Note that the $\varepsilon$-twisted conjugation of a positive braid by a positive braid yields a positive braid. Also,
note that, for any braid $x\in B_n$ and any generator $\sigma_i$, $x$ and $\sigma_i x\sigma_i$ are
$\varepsilon$-twisted conjugated. Imposing positivity, this yields to the following definition:

\begin{definition}\label{D:palindromic_free}
A positive braid $x$ is said to be {\bf palindromic-free} if $\sigma_i^{-1} x \sigma_i^{-1}$ is not positive, for every
$i=1,\ldots,n-1$.
\end{definition}

In other words, if $x$ is a positive, palindromic-free braid and $\sigma_i\preccurlyeq x$, that is, $x=\sigma_i y$ for
some positive $y$, then $y \not\succcurlyeq \sigma_i$. However, notice that even if $x$ is palindromic-free, one may
have simultaneously $\sigma_i\preccurlyeq x$ and $x\succcurlyeq \sigma_i$ for some $i$. For instance if $x=\sigma_i$,
or if $x=\sigma_i\sigma_j$ with $|i-j|\geqslant 2$.

\begin{proposition}\label{P:MPF_nonempty}
Every braid $x\in B_n$ is $\varepsilon$-twisted conjugated to some positive, palindromic-free braid $y$.
\end{proposition}

\begin{proof}
It is well known that for every braid $x\in B_n$, the braid $x \Delta^p $ is positive for $p$ big enough. Since
$\overleftarrow \Delta $ is positive (actually $\overleftarrow \Delta = \Delta$), it follows that
$\overleftarrow{\Delta^p} x \Delta^p$ is positive for some $p$ big enough. Hence $x$ is $\varepsilon$-twisted
conjugated to a positive braid $z$.

If $z$ is not palindromic-free, there will exist a letter $\sigma_i$ such that $z=\sigma_i z' \sigma_i$ for some
positive braid $z'$ whose word length is smaller than that of $z$. And, since $\overleftarrow{\sigma_i} =\sigma_i$, $z$
is $\varepsilon$-twisted conjugated to $z'$. Repeating this process, as the word length of the resulting braid
decreases at each step, one finally obtains a palindromic-free positive braid $\varepsilon$-twisted conjugated to $z$,
thus to $x$.
\end{proof}

By the above argument, every positive braid $x$ has the form $x=\overleftarrow c yc$ for some positive,
palindromic-free braid $y$. We remark that the element $y$ is not unique. For instance, if $x=\sigma_2\sigma_1\sigma_2
= \sigma_1 \sigma_2 \sigma_1$, then $y$ could be equal to either $\sigma_1$ or $\sigma_2$. Another example is
$x=\sigma_3 \sigma_2 \sigma_1 \sigma_2 \sigma_3 = \sigma_3 \sigma_1 \sigma_2 \sigma_1 \sigma_3 = \sigma_1 \sigma_2
\sigma_3 \sigma_2 \sigma_1$, so $y$ could be equal, in this case, to either $\sigma_1$ or $\sigma_2$ or $\sigma_3$.

Recall that we are trying to find an algorithm to solve the $\varepsilon$-twisted conjugacy problem in $B_n$. After the
above discussion, one may think that a possible solution could be to compute the set of positive, palindromic-free
braids, $\varepsilon$-twisted conjugated to a given one. Clearly, two braids $u$ and $v$ are $\varepsilon$-twisted
conjugated if and only if their corresponding sets coincide. Unfortunately, this attempt does not work because the
mentioned set is not always finite, as one can see in the following example.

\begin{example}
The set $\{\sigma_3^n \sigma_2 \sigma_3 \sigma_4 \sigma_5 \sigma_1 \sigma_2 \sigma_3 \sigma_4  \sigma_1^n; \quad n\geq
0\}\subset B_6$ is an infinite family of positive, palindromic-free braids, which are pairwise $\varepsilon$-twisted
conjugated.
\end{example}

\begin{proof}
We will show that for every $n\geq 0$ one has:
 $$
\sigma_5^n\left(\sigma_3^n \sigma_2 \sigma_3 \sigma_4 \sigma_5 \sigma_1 \sigma_2 \sigma_3 \sigma_4
\sigma_1^n\right)\sigma_5^n = (\sigma_1\sigma_3)^n \sigma_2 \sigma_3 \sigma_4 \sigma_5 \sigma_1 \sigma_2 \sigma_3
\sigma_4 (\sigma_3\sigma_1)^n.
 $$
So all braids in the above family are $\varepsilon$-twisted conjugated to $\sigma_2 \sigma_3 \sigma_4 \sigma_5 \sigma_1
\sigma_2 \sigma_3 \sigma_4$, and so to each other. To see this, first notice that
\begin{eqnarray*}
   \sigma_5 \left(\sigma_2 \sigma_3 \sigma_4 \sigma_5 \sigma_1 \sigma_2 \sigma_3 \sigma_4 \right) & = &
   \left(\sigma_2 \sigma_3 \right) \sigma_5 \left(\sigma_4 \sigma_5 \sigma_1 \sigma_2 \sigma_3 \sigma_4\right) \\
  & = &    \left(\sigma_2 \sigma_3 \sigma_4 \sigma_5 \right) \sigma_4 \left(\sigma_1 \sigma_2 \sigma_3 \sigma_4\right) \\
  & = &    \left(\sigma_2 \sigma_3 \sigma_4 \sigma_5 \sigma_1\sigma_2 \right) \sigma_4 \left(\sigma_3 \sigma_4\right) \\
  & = &    \left(\sigma_2 \sigma_3 \sigma_4 \sigma_5 \sigma_1\sigma_2 \sigma_3 \sigma_4 \right) \sigma_3.
\end{eqnarray*}
On the other hand, by commutativity relations,
$$
  \sigma_2 \sigma_3 \sigma_4 \sigma_5 \sigma_1 \sigma_2 \sigma_3 \sigma_4 =
  \sigma_2 \sigma_1 \sigma_3  \sigma_2 \sigma_4 \sigma_3 \sigma_5 \sigma_4,
$$
hence
\begin{eqnarray*}
   \sigma_1 \left(\sigma_2 \sigma_3 \sigma_4 \sigma_5 \sigma_1 \sigma_2 \sigma_3 \sigma_4 \right) & = &
   \sigma_1 \left(\sigma_2 \sigma_1 \sigma_3  \sigma_2 \sigma_4 \sigma_3 \sigma_5 \sigma_4 \right) \\
  & = &    \left(\sigma_2 \sigma_1 \right) \sigma_2 \left(\sigma_3 \sigma_2 \sigma_4 \sigma_3 \sigma_5 \sigma_4\right) \\
  & = &    \left(\sigma_2 \sigma_1 \sigma_3 \sigma_2 \right) \sigma_3 \left(\sigma_4 \sigma_3 \sigma_5 \sigma_4\right) \\
  & = &    \left(\sigma_2 \sigma_1 \sigma_3 \sigma_2 \sigma_4 \sigma_3 \right) \sigma_4 \left(\sigma_5 \sigma_4\right) \\
  & = &    \left(\sigma_2 \sigma_1 \sigma_3 \sigma_2 \sigma_4 \sigma_3 \sigma_5 \sigma_4 \right) \sigma_5   \\
  & = &    \left(  \sigma_2 \sigma_3 \sigma_4 \sigma_5 \sigma_1 \sigma_2 \sigma_3 \sigma_4 \right) \sigma_5.
\end{eqnarray*}
Therefore, as $\sigma_1$, $\sigma_3$ and $\sigma_5$ commute, one has
\begin{eqnarray*}
  \sigma_5^n \sigma_3^n \left(\sigma_2 \sigma_3 \sigma_4 \sigma_5 \sigma_1 \sigma_2 \sigma_3 \sigma_4 \right) \sigma_1^n \sigma_5^n
  & = & \sigma_3^n \left(\sigma_2 \sigma_3 \sigma_4 \sigma_5 \sigma_1 \sigma_2 \sigma_3 \sigma_4 \right)\sigma_3^{n} \sigma_1^n \sigma_5^n \\
  & = & \sigma_1^n \sigma_3^n \left(\sigma_2 \sigma_3 \sigma_4 \sigma_5 \sigma_1 \sigma_2 \sigma_3 \sigma_4 \right)\sigma_3^{n} \sigma_1^n,
\end{eqnarray*}
and the claim follows.

It just remains to show that every element in the above family is palindromic-free. This could be easily done by using
the standard topological representation of braids as collections of strands in $\mathbb R^3$, but we will show it
algebraically.

If $n=0$, we have the braid $\alpha_0=\sigma_2 \sigma_3 \sigma_4 \sigma_5 \sigma_1 \sigma_2 \sigma_3 \sigma_4$. We
recall that the monoid $B_6^+$ of positive braids embeds in $B_6$, so we just need to use positive relations from the
standard presentation~(\ref{E:presentation}) to determine which generators are prefixes or suffixes of $\alpha_0$.  But
notice that in the above word, no matter how many commutativity relations we apply, we can never obtain a subword of
the form $\sigma_i\sigma_j\sigma_i$, because between two appearances of the letter $\sigma_i$ one always has both
$\sigma_{i-1}$ and $\sigma_{i+1}$. Hence, only commutativity relations can be applied, and it follows that this braid
can only start with $\sigma_2$, and can only end with $\sigma_4$, thus it is palindromic-free.

For $n>0$, the braid we are considering is $\alpha_n=\sigma_3^n (\sigma_2 \sigma_3 \sigma_4 \sigma_5 \sigma_1 \sigma_2
\sigma_3 \sigma_4) \sigma_1^n$. Notice that:
\begin{eqnarray*}
   \sigma_3 \left(\sigma_2 \sigma_3 \sigma_4 \sigma_5 \sigma_1 \sigma_2 \sigma_3 \sigma_4 \right) & = &
   \left(\sigma_2 \sigma_3 \right) \sigma_2 \left(\sigma_4 \sigma_5 \sigma_1 \sigma_2 \sigma_3 \sigma_4\right) \\
  & = &    \left(\sigma_2 \sigma_3 \sigma_4 \sigma_5 \right) \sigma_2 \left(\sigma_1 \sigma_2 \sigma_3 \sigma_4\right) \\
  & = &    \left(\sigma_2 \sigma_3 \sigma_4 \sigma_5 \sigma_1\sigma_2 \right) \sigma_1 \left(\sigma_3 \sigma_4\right) \\
  & = &    \left(\sigma_2 \sigma_3 \sigma_4 \sigma_5 \sigma_1\sigma_2 \sigma_3 \sigma_4 \right) \sigma_1.
\end{eqnarray*}
Hence $\alpha_n=\sigma_3^{2n}(\sigma_2 \sigma_3 \sigma_4 \sigma_5 \sigma_1 \sigma_2 \sigma_3 \sigma_4)$, and also
$\alpha_n=(\sigma_2 \sigma_3 \sigma_4 \sigma_5 \sigma_1 \sigma_2 \sigma_3 \sigma_4)\sigma_1^{2n}$.

On one hand, the above two expressions of $\alpha_n$ show that it can start with $\sigma_3$ and also with $\sigma_2$.
Suppose that it can also start with $\sigma_1$. As $\sigma_3^{2n}\preccurlyeq \alpha_n$ and we are assuming that
$\sigma_1\preccurlyeq \alpha_n$, it follows that $\sigma_3^{2n}\vee \sigma_1 \preccurlyeq \alpha_n$, that is
$\sigma_3^{2n}\sigma_1\preccurlyeq \alpha_n$. Multiplying by $\sigma_3^{-2n}$ from the left we obtain
$\sigma_1\preccurlyeq \sigma_2 \sigma_3 \sigma_4 \sigma_5 \sigma_1 \sigma_2 \sigma_3 \sigma_4$. But this is not
possible as the latter braid can only start with $\sigma_2$. Hence $\sigma_1\not\preccurlyeq \alpha_n$. Analogously, as
$\sigma_3^{2n}\vee \sigma_4 = \sigma_3^{2n}\sigma_4 \sigma_3$, and also $\sigma_3^{2n}\vee \sigma_5 =
\sigma_3^{2n}\sigma_5$, it follows that $\sigma_4\not\preccurlyeq \alpha_n$ and $\sigma_5\not\preccurlyeq \alpha_n$.
Therefore $\alpha_n$ can only start with either $\sigma_2$ or $\sigma_3$.

The symmetric argument shows that $\alpha_n$ can only end with either $\sigma_1$ or $\sigma_4$. Therefore $\alpha_n$ is
palindromic-free, as we wanted to show.
\end{proof}

Hence, the attempt to compute the set of all positive, palindromic-free braids, $\varepsilon$-twisted conjugated to a
given one does not work. However, we shall save the idea by imposing a further condition which will assure the required
finiteness of the set: we shall consider only elements with minimal canonical length. The set we will compute is then
the following.

\begin{definition}
Given a braid $x\in B_n$, we define $MPF(x)$ to be the set of positive, palindromic-free braids, $\varepsilon$-twisted
conjugated to $x$, of minimal canonical length.
\end{definition}

Notice that if a positive braid $x$ is palindromic-free, then $\inf(x)=0$, so $\sup(x)=\ell(x)$. This gives us
finiteness of $MPF(x)$:

\begin{proposition}
For every $x\in B_n$, the set $MPF(x)$ is nonempty and finite, and it is an invariant of its $\varepsilon$-twisted
conjugacy class.
\end{proposition}

\begin{proof}
$MPF(x)$ is an invariant of the $\varepsilon$-twisted conjugacy class of $x$ by definition. It is nonempty by
Proposition~\ref{P:MPF_nonempty}, and it is finite since the set of elements of infimum zero, and given canonical
length, is finite.
\end{proof}

\section{The twisted conjugacy problem for $B_n$.}

In order to find a solution to the $\varepsilon$-twisted conjugacy problem in braid groups, we need a method to compute
$MPF(x)$, given $x\in B_n$. For that purpose, we shall need the following technical results.

\begin{lemma}\label{L:increasing_inf}
Let $u$ be a positive, palindromic-free braid. Let $c$ be a positive braid with $\inf(c)=0$, whose left normal form is
$c=c_1\cdots c_s$. Denote $k_i=\inf(\overleftarrow{c_i}\cdots \overleftarrow{c_1} \: u \: c_1\cdots c_i)$. Then
$k_{i+1}\leq k_i+1$ for $i=0,\ldots,s-1$. In particular, $\inf(\overleftarrow{c_i}\cdots \overleftarrow{c_1} \: u \:
c_1\cdots c_i)\leq i$ for $i=0,\ldots,s$.
\end{lemma}

\begin{proof} Recall that, since $u$ is palindromic-free, $\inf(u)=0$. As the infimum of an element can increase by at
most one when multiplied by a simple element, one has either $\inf(uc_1)=0$ or $\inf(uc_1)=1$.

Suppose that $\inf(uc_1)=0$, that is, $\Delta$ is not a prefix of $uc_1$. It is well known that, as $c_1\cdots c_s$ is
in left normal form, then $\inf(uc_1\cdots c_s)=0$. Since the infimum of an element can increase by at most one when it
is multiplied by a simple element, one has $\inf(\overleftarrow{c_i}\cdots  \overleftarrow{c_1} \: u \: c_1\cdots
c_s)\leq i$, moreover $\inf(\overleftarrow{c_{i}}\cdots  \overleftarrow{c_1} \: u \: c_1\cdots c_s)\leq
\inf(\overleftarrow{c_{i-1}}\cdots  \overleftarrow{c_1} \: u \: c_1\cdots c_s)+1$ for $i=1,\ldots,s$. It suffices then
to show that $k_i=\inf(\overleftarrow{c_i}\cdots  \overleftarrow{c_1} \: u \: c_1\cdots c_s)$ for $i=0,\ldots,s$. But
since we already showed that $\inf(\overleftarrow{c_i}\cdots  \overleftarrow{c_1} \: u \: c_1\cdots c_s)\leq i$, and
$c_1\cdots c_s$ is in left normal form, then $\Delta^p$ is a prefix of $\overleftarrow{c_i}\cdots \overleftarrow{c_1}
\: u \: c_1\cdots c_s$ (necessarily $p\leq i$) if and only if it is a prefix of $\overleftarrow{c_i}\cdots
\overleftarrow{c_1} \: u \: c_1\cdots c_i$. Hence $\inf(\overleftarrow{c_i}\cdots \overleftarrow{c_1} \: u \: c_1\cdots
c_s)= \inf(\overleftarrow{c_i}\cdots  \overleftarrow{c_1} \: u \: c_1\cdots c_i)=k_i$ for $i=0,\ldots,s$, as we wanted
to show.

Now suppose that $\inf(uc_1)=1$, that is, $uc_1= v\Delta$ for some positive $v$, prefix of $u$. This means that
$u=v\partial^{-1}(c_1)$. Since $u$ is palindromic-free, one has \mbox{$\overleftarrow{\partial^{-1}(c_1)}\wedge v =1$}.
But it is easy to see that $\overleftarrow{\partial^{-1}(c_1)}=\partial(\overleftarrow{c_1})$, so one has
\mbox{$\partial(\overleftarrow{c_1})\wedge v=1$}, that is, the decomposition $\overleftarrow{c_1} v$ is left-weighted
as written. This in particular implies that $\inf(\overleftarrow{c_1}v)=0$ and, since $\overleftarrow{c_s}\cdots
\overleftarrow{c_1}$ is in right normal form, that $\inf(\overleftarrow{c_s}\cdots \overleftarrow{c_1}v)=0$. Hence
$\inf(\overleftarrow{c_s}\cdots \overleftarrow{c_1}uc_1)= \inf(\overleftarrow{c_s}\cdots
\overleftarrow{c_1}v\Delta)=1=k_1=k_0+1$. As the infimum can increase by at most one when an element is multiplied by a
simple one, then one has $\inf(\overleftarrow{c_s}\cdots \overleftarrow{c_1}uc_1\cdots c_i)\leq i$, moreover
$\inf(\overleftarrow{c_s}\cdots \overleftarrow{c_1}uc_1\cdots c_{i})\leq \inf(\overleftarrow{c_s}\cdots
\overleftarrow{c_1}uc_1\cdots c_{i-1})+1$ for $i=1,\ldots,s$. Repeating the argument of the previous case, one has
$k_i=\inf(\overleftarrow{c_s}\cdots \overleftarrow{c_1} \: u \: c_1\cdots c_i)$ for $i=0,\ldots,s$, and the result is
shown.
\end{proof}

\begin{corollary}\label{C:increasing_inf}
Let $u$ be a positive, palindromic-free braid. Let $c$ be a positive braid with $\inf(c)=0$, and whose left normal form
is $c=c_1\cdots c_s$. If $\inf(\overleftarrow{c_s}\cdots \overleftarrow{c_1} \: u \: c_1\cdots c_s)=s$, then
$\inf(\overleftarrow{c_i}\cdots \overleftarrow{c_1} \: u \: c_1\cdots c_i)=i$ for $i=0,\ldots,s$.
\end{corollary}

\begin{proof}
Let $k_i=\inf(\overleftarrow{c_i}\cdots \overleftarrow{c_1} \: u \: c_1\cdots c_i)$ for $i=0,\ldots,s$. We know that
$k_0=0$ since $u$ is palindromic-free, and that $k_{i+1}\leq k_i+1$ by the previous result. By induction on $s$, it
follows that $k_s\leq s$ and the equality holds if and only if $k_{i+1}=k_i+1$ for $i=0,\ldots,s-1$. But we have
$k_s=s$ by hypothesis, hence $k_i=i$ for $i=0,\ldots,s$, as we wanted to show.
\end{proof}

\begin{corollary}\label{C:equal_length}
Let $u,v\in B_n$ be positive, palindromic-free braids. Let $a$ and $b$ be nontrivial positive braids such that
$a\wedge_R b=1$ (hence $\inf(a)=\inf(b)=0$). Suppose that $\overleftarrow{a} u a = \overleftarrow{b} v b$. Then
$\ell(a)=\ell(b)$.
\end{corollary}

\begin{proof}
Denote $\ell(a)=p$ and $\ell(b)=q$, and write $a=a_1\cdots a_p$ and $b=b_1\cdots b_q$ in right normal forms. Consider
$b^*=b^{-1}\Delta^q$. Then $b^*$ is a positive braid with $\inf(b^*)=0$. Namely, its right normal form is $b^*=
\partial(b_q) \partial^3(b_{q-1})\cdots \partial^{2q-1}(b_1)$. Then consider the product
 $$
ab^*=a_1\cdots a_p \partial(b_q) \partial^3(b_{q-1})\cdots \partial^{2q-1}(b_1).
 $$
We claim that the above decomposition is the right normal form of $ab^*$. We just need to show that $a_p \partial(b_q)$
is right-weighted as written. But $a\wedge_R b=1$, so $1=a_p \wedge_R b_q = a_p \wedge_R \partial^{-1}(\partial(b_q))$,
which precisely means that $a_p \partial(b_q)$ is right-weighted, showing the claim.  This implies in particular that
$\inf(ab^*)=0$ and $\ell(ab^*)=p+q$.

Notice that $\overleftarrow{b^*} \overleftarrow{a} u ab^* = \overleftarrow{b^*} \overleftarrow{b} v b b^* = \Delta^q v
\Delta^q$. Since $\inf(v)=0$ as $v$ is palindromic free, one has $\inf(\overleftarrow{b^*} \overleftarrow{a} u
ab^*)=\inf(\Delta^q v \Delta^q)=2q$. On the other hand, $\inf(ab^*)=0$ and $\ell(ab^*)=p+q$, so
Lemma~\ref{L:increasing_inf} implies that $\inf(\overleftarrow{b^*} \overleftarrow{a} u ab^*)\leq p+q$. Therefore
$2q\leq p+q$, that is, $q\leq p$. By symmetry, one also has $p\leq q$, so the equality holds.
\end{proof}

Recall that we want to find a method to compute, for any given braid $x\in B_n$, the set $MPF(x)$ i.e. the (finite) set
of positive, palindromic-free, $\varepsilon$-twisted conjugates of $x$ of minimal canonical length. Notice that if two
elements $u$ and $v$ are $\varepsilon$-twisted conjugated, that is, if $\overleftarrow{c}uc=v$ for some braid $c$, then
we can multiply on both sides by a suitable power of $\Delta$ such that $c\Delta^p$ is positive, in such a way that
$\Delta^p \overleftarrow c u c \Delta^p =\Delta^p v \Delta^p $, so we have written $\overleftarrow A u A =
\overleftarrow B v B$ with $A$ and $B$ positive. Moreover, if $d=A\wedge_R B$ is the maximal common suffix of $A$ and
$B$, then multiplying the above equality from the right by $d^{-1}$ and from the left by $(\overleftarrow{d})^{-1}$, we
finally get $\overleftarrow a u a = \overleftarrow b v b$, with $a$ and $b$ positive and such that $a\wedge_R b =1$, as
in the hypothesis of the above result.  We will be specially interested in the case in which $a$ and $b$ are simple
elements.

\begin{definition}
We will say that two elements $u,v\in B_n$ are {\bf simply $\varepsilon$-twisted conjugated}, or that they are related
by a {\bf simple $\varepsilon$-twisted conjugation}, if there exist simple elements $a$ and $b$ such that
$\overleftarrow a u a = \overleftarrow b v b$.
\end{definition}

The main result of this section is analogous, with respect to $\varepsilon$-twisted conjugacy, to the following famous
result by El-Rifai and Morton with respect to conjugacy.

\begin{theorem}\label{T:EM}{\rm \cite{EM}}
Let $u,v\in B_n$ be conjugated braids such that $\ell(u)\leq r$ and $\ell(v)\leq r$ for some $r$. Then there is a chain
$u=u_0, u_1,\ldots, u_k=v$, with $\ell(u_i)\leq r$ for all $i$, such that $u_{i-1}$ is conjugated to $u_i$ by a simple
element, for $i=1,\ldots,k$. Namely, if $c$ is a positive element such that $c^{-1}uc=v$, and $c=c_1\cdots c_k$ is its
left normal form, then one can take $u_i= c_i^{-1}\cdots c_1^{-1} u c_1\cdots c_i$.
\end{theorem}

In our case, dealing with $\varepsilon$-twisted conjugacy, we will restrict to positive, palin\-dro\-mic-free braids.

\begin{theorem}\label{T:twisted_EM}
Let $u,v\in B_n$ be positive, palindromic-free, $\varepsilon$-twisted conjugated braids such that $\ell(u)\leq r$ and
$\ell(v)\leq r$ for some $r$. Then there is a chain $u=u_0, u_1,\ldots, u_k=v$ of positive, palindromic-free braids,
with $\ell(u_i)\leq r$ for all $i$, such that $u_{i-1}$ is simply $\varepsilon$-twisted conjugated to $u_i$, for
$i=1,\ldots,k$.
\end{theorem}

\begin{proof}
As we saw above, there are positive elements $a$ and $b$, with $a\wedge _R b=1$, such that $\overleftarrow a u a =
\overleftarrow b v b$. Since $u$ and $v$ are palindromic free, $a$ is trivial if and only if so is $b$. If $a$ and $b$
are nontrivial, the hypotheses of Corollary~\ref{C:equal_length} are satisfied, thus $\ell(a)=\ell(b)=p$ in any case.
We will show the result by induction on $p$. If $p=0$ the result is trivially true, so we will assume that $p>0$ and
that the result is true for all values between 0 and $p-1$.

The strategy of the proof will be to find some palindromic-free braid $w$ with $\ell(w)\leq r$, such that
$\overleftarrow s u s= \overleftarrow t w t$ for some simple braids $s$ and $t$ (this is a chain of length 1 from $u$
to $w$), and also $\overleftarrow{y} w y = \overleftarrow{z} v z$ for some positive elements $y,z$ such that $y\wedge_R
z=1$ and $\ell(y)=\ell(z)\leq p-1$. The induction hypothesis provides a chain from $w$ to $v$, so the result will
follow by concatenating both chains.

We start as in the proof of  Corollary~\ref{C:equal_length}, defining $b^*=b^{-1}\Delta^p$, and noticing that
$\inf(ab^*)=0$, $\ell(ab^*)=2p$ and $\overleftarrow{b^*}\overleftarrow{a} u a b^* = \Delta^p v \Delta^p$. Denote
$c=ab^*$, and let $c=c_1\cdots c_{2p}$ be its left normal form. Then $\overleftarrow{c}uc=\Delta^p v \Delta^p$. Since
$v$ is palindromic-free, thus $\inf(v)=0$, one has $\inf(\overleftarrow{c}uc)=2p$. By Corollary~\ref{C:increasing_inf}
one has $\inf(\overleftarrow{c_i}\cdots \overleftarrow{c_1} u c_1\cdots c_i)=i$ for $i=1,\ldots,2p$.  In particular
$\inf(\overleftarrow{c_2}\overleftarrow{c_1} u c_1c_2)=2$, hence $\overleftarrow{c_2}\overleftarrow{c_1} u c_1c_2 =
\Delta w'\Delta$ for some positive braid $w'$.

Multiplying the above equality on the right by $c_2^{-1}$ and on the left by its reverse, we obtain
$\overleftarrow{c_1} u c_1 = \overleftarrow{\partial^{-1}(c_2)}  w' \partial^{-1}(c_2)$. Hence $u$ and $w'$ are simply
$\varepsilon$-twisted conjugated. But $w'$ is not necessarily palindromic-free, and one does not necessarily have
$\ell(w')\leq r$. Let us see that we can replace $w'$ by some $w$ that satisfies the required hypothesis.

Recall that $\overleftarrow{c_2}\overleftarrow{c_1} u c_1c_2 = \Delta w'\Delta$. Since the left hand side is a product
of at most $r+4$ simple elements, it follows that $\sup(w')\leq r+2$. Moreover, multiplying each side of the equality,
from the left, by its image under $\varepsilon$, one has
 $$
\left(\varepsilon(\overleftarrow{c_2}\overleftarrow{c_1} u c_1c_2)\right) \left( \overleftarrow{c_2}\overleftarrow{c_1}
u c_1c_2 \right)= \left(\varepsilon(\Delta w'\Delta)\right)\left(\Delta w'\Delta\right).
 $$
Hence:
 $$
\left(c_2^{-1} c_1^{-1} \varepsilon(u) \varepsilon(c_1) \varepsilon(c_2) \right) \left(\overleftarrow{c_2}
\overleftarrow{c_1} u c_1c_2 \right)= \left(\Delta^{-1} \varepsilon(w') \Delta^{-1} \right)\left(\Delta w' \Delta
\right).
 $$
Since $\varepsilon(c_i)= (\overleftarrow{c_i})^{-1}$, one obtains:
 $$
c_2^{-1} c_1^{-1} (\varepsilon(u)u) c_1 c_2 = \tau(\varepsilon(w')w').
 $$
In the same way, from the equality
 $$
\left(\overleftarrow{c_{2p}} \cdots \overleftarrow{c_1}\right) u \left( c_1 \cdots c_{2p} \right) = \Delta^p v
\Delta^p,
 $$
one gets:
 $$
\left( c_{2p}^{-1} \cdots c_1^{-1}\right) \varepsilon(u)u \left( c_1 \cdots c_{2p}\right) = \tau^{p}(\varepsilon(v)v).
 $$
Recall that $\ell(u)\leq r$ and $\ell(v)\leq r$, so by Corollary~\ref{C:pal-free_double_length} one has
$\ell(\varepsilon(u)u)\leq 2r$ and $\ell(\varepsilon(v)v)\leq 2r$,  thus $\ell(\tau^p(\varepsilon(v)v))\leq 2r$.
Therefore, by Theorem~\ref{T:EM}, $\ell(\tau(\varepsilon(w')w'))= \ell(c_2^{-1} c_1^{-1} \varepsilon(u)u c_1 c_2)\leq
2r$. Hence $\ell(\varepsilon(w')w')\leq 2r$.

We claim that there are positive braids $x$ and $w$, such that $w'=\overleftarrow{x}wx$ and $\sup(w)\leq r$. First, if
$\sup(w')\leq r$ one can take $x=1$ and $w=w'$. Second, if $\sup(w')=r+1$, notice that $r\geq 1$ and
$\ell(\varepsilon(w')w')\leq 2r$, so the claim follows from Lemma~\ref{L:sup_and_palindromic}, taking $k=1$. We must
then show the claim in the case $\sup(w')=r+2$.

Suppose that $\sup(w')=r+2$, and recall that $\ell(\varepsilon(w')w')\leq 2r = 2\sup(w')-4$. If $\sup(w')\geq 4$, the
claim follows from Lemma~\ref{L:sup_and_palindromic}, taking $k=2$. Therefore the only remaining case is $\sup(w')=3$,
$r=1$ and $\ell(\varepsilon(w')w')\leq 2$.  Let $d=w'\wedge \mbox{rev}(w')$ and write $w'=d \alpha $ and
$\mbox{rev}(w')=d \beta$. Notice that $\varepsilon(w')w' = \mbox{rev}(w')^{-1} w' = \beta^{-1}d^{-1}d\alpha$, hence the
mixed normal form of $\varepsilon(w')w'$ is precisely $\beta^{-1}\alpha$.  Moreover, since the word length of $w'$ and
$\mbox{rev}(w')$ coincide, one has $\alpha=1$ if and only if $\beta=1$. Hence, since $\sup(\alpha)+\sup(\beta)
=\ell(\varepsilon(w')w')  \leq 2$, one must necessarily have either $\sup(\alpha)=\sup(\beta)=0$ or
$\sup(\alpha)=\sup(\beta)=1$, that is, $\alpha$ and $\beta$ are (possibly trivial) simple elements.

Write $w'=a_1a_2a_3$ in left normal form. The right normal form of $\mbox{rev}(w')$ is then $\overleftarrow{a_3}
\overleftarrow{a_2} \overleftarrow{a_1}$.  Since $\mbox{rev}(w')=d\beta$ and $\beta$ is simple, it follows that
$\overleftarrow{a_3}\overleftarrow{a_2}\preccurlyeq d$, hence $\overleftarrow{a_3}\overleftarrow{a_2}\preccurlyeq
d\alpha = w' = a_1a_2a_3$. Since the latter decomposition is in left normal form, one has $\overleftarrow{a_3}
\overleftarrow{a_2} \preccurlyeq a_1 a_2$, and also $\overleftarrow{a_3}\preccurlyeq a_1$. Write then $w'=
\overleftarrow{a_3} (c a_2)a_3$ for some positive $c$. Now if $c a_2$ is simple we are done, as one can take $x=a_3$
and $w=ca_2$. Otherwise, write $ca_2= b_1b_2$ in left normal form, and recall that $\overleftarrow{a_3}
\overleftarrow{a_2} \preccurlyeq a_1 a_2$, so $\overleftarrow{a_2}\preccurlyeq c a_2 = b_1b_2$. Then
$\overleftarrow{a_2}\preccurlyeq b_1$. On the other hand, since $ca_2=b_1b_2$ and the latter decomposition is left
weighted, one has $a_2 \succcurlyeq b_2$ and then $\overleftarrow{b_2}\preccurlyeq \overleftarrow{a_2}$. Concatenating
the last two inequalities, one finally obtains $\overleftarrow{b_2}\preccurlyeq \overleftarrow{a_2} \preccurlyeq b_1$.
Therefore one can write $b_1= \overleftarrow{b_2}w$ for some simple element $w$, and one has $w'= (\overleftarrow{a_3}
\overleftarrow{b_2}) w (b_2 a_3)$. Taking $x=b_2a_3$, the claim is shown.

Notice that if $w$ is not palindromic-free, we can still decompose $w=\overleftarrow{y}w''y$ where $y$ is positive and
$w''$ is palindromic-free. Moreover, $\sup(w'')\leq \sup(w)\leq r$. Therefore, replacing $x$ by  $yx$ and $w$ by $w''$
if necessary, we can assume that $w'=\overleftarrow{x}wx$, where $x$ is positive and $w$ is palindromic-free with
$\sup(w)\leq r$.

Now recall that $\overleftarrow{c_1} u c_1 = \overleftarrow{\partial^{-1}(c_2)}  w' \partial^{-1}(c_2)$ for simple
elements $c_1$ and $c_2$. By the above claim, $\overleftarrow{c_1} u c_1 = \left(\overleftarrow{\partial^{-1}(c_2)}
\overleftarrow{x}\right) w \left(x \partial^{-1}(c_2)\right)$. Multiplying this equality from the right by
$\left(c_1\wedge_R \left(x \partial^{-1}(c_2)\right)\right)^{-1}$ and from the left by its reverse, we obtain
$\overleftarrow{s}us = \overleftarrow{t}wt$ for positive braids $s$ and $t$ such that $s\wedge_R t=1$. Now $s$ is
simple as it is a prefix of $c_1$, hence $t$ is simple by Corollary~\ref{C:equal_length}. Therefore $u$ and $w$ are
simply $\varepsilon$-twisted conjugated, positive, palindromic-free braids, whose canonical length is at most $r$. This
is the first step of our required chain.

Now notice that
\begin{eqnarray*}
\overleftarrow{c_{2p}} \cdots \overleftarrow{c_3} \Delta  w' \Delta c_3\cdots c_{2p}
& = & \overleftarrow{c_{2p}} \cdots \overleftarrow{c_3} \overleftarrow{c_2} \overleftarrow{\partial^{-1}(c_2)}  w' \partial^{-1}(c_2) c_2 c_3\cdots c_{2p} \\
& = &  \overleftarrow{c_{2p}} \cdots \overleftarrow{c_2} \overleftarrow{c_1}  u c_1 c_2\cdots c_{2p} \\
& = & \Delta^p   v \Delta^p.
\end{eqnarray*}
Hence
 $$
{\tau^{-1}(\overleftarrow{c_{2p}}\cdots \overleftarrow{c_{3}})} \: w' \: \tau^{-1}(c_3\cdots c_{2p}) = \Delta^{p-1} v
\Delta^{p-1}.
 $$
For simplicity, we will denote $d_i=\tau^{-1}(c_{i+2})$ for $i=1,\ldots,2p-2$. Hence we have:
 $$
\left(\overleftarrow{d_{2p-2}}\cdots \overleftarrow{d_1}\right) \: w' \: \left(d_1\cdots d_{2p-2} \right)= \Delta^{p-1}
v \Delta^{p-1}.
 $$
Recalling that $w'=\overleftarrow{x}wx$, and multiplying the above equality from the right by $(d_p\cdots
d_{2p-2})^{-1}$ and from the left by its reverse, we finally obtain:
 $$
\left(\overleftarrow{d_{p-1}}\cdots \overleftarrow{d_1} \overleftarrow{x}\right) \: w \: \left(x d_1\cdots d_{p-1}
\right)= \left(\overleftarrow{e_{p-1}}\cdots \overleftarrow{e_1}\right) \: v \: \left(e_1\cdots e_{p-1} \right),
 $$
where $e_1,\ldots, e_{p-1}$ are simple elements and $e_1\cdots e_{p-1} = \Delta^{p-1} \left(d_p\cdots
d_{2p-2}\right)^{-1}$. Reducing the above equality, if necessary, by the biggest common suffix of $\left(x d_1\cdots
d_{p-1} \right)$ and $\left(e_1\cdots e_{p-1} \right)$, it follows that there exist positive braids $y$ and $z$ such
that $y\wedge_R z=1$, and $\overleftarrow{y} w y = \overleftarrow{z} v z$. Recall that $w$ and $v$ are palindromic-free
and, by Corollary~\ref{C:equal_length} and as $z$ is a prefix of $e_1\cdots e_{p-1}$, $\ell(y)=\ell(z)\leq p-1$.
Therefore the induction hypothesis provides the remaining part of the required chain, and the result is shown.
\end{proof}

\begin{corollary}\label{bottom} Let $u,v\in B_n$ be positive,
palindromic-free, $\varepsilon$-twisted conjugated braids of minimal canonical length in their $\varepsilon$-twisted
conjugacy class, say $r=\ell(u)=\ell(v)$. Then there is a chain $u=u_0, u_1,\ldots, u_k=v$ of positive,
palindromic-free braids, with canonical length $\ell(u_i)=r$ for all $i$, such that $u_{i-1}$ is simply
$\varepsilon$-twisted conjugated to $u_i$, for $i=1,\ldots,k$.
\end{corollary}

\begin{corollary}\label{computeMPF}
There exists an algorithm to compute $MPF(x)$ for any given braid $x\in B_n$.
\end{corollary}

\begin{proof}
If $x=1$ then $MPF(x)=\{ 1\}$. So, let us assume $x\neq 1$.

First of all, compute a positive, palindromic-free, $\varepsilon$-twisted conjugate of $x$, say $y$, as it is explained
in Proposition~\ref{P:MPF_nonempty}. Let $r=\ell (y)\geqslant 1$, and let $S=\{ y \} \subset B_n$.

Now, consider the following operation, which will have to be subsequently applied  until all elements in $S$ have been
processed:

\emph{Choose $z\in S$ which has not been processed, compute all positive palindromic-free elements which are simply
$\varepsilon$-twisted conjugated to $z$ and have canonical length less than or equal to $r$ (this is clearly a finite,
computable set), and then do the following: 1) if one of them, say $z'$, has length less than $r$, kill the whole
process, reset $y=z'$, $S=\{ z'\}$, $r=\ell(z')$ and start the algorithm again; 2) otherwise, add to $S$ all the
computed elements (which have canonical length exactly equal to $r$), and mark $z$ as processed.}

At each application of such operation, either the set $S$ gets restarted and $r$ strictly decreased, or the set $S$
gets increased by the addition of the new elements computed (some of which could already be present in the former $S$).
But $r\geqslant 1$ can only decrease a finite number of times, and $|S|$ can only increase a finite number of times,
since the number of braids with infimum zero and given canonical length is finite (recall that palindromic-free
elements have infimum zero).

Hence, after a finite number of applications of the previous operation (running over all elements $z\in S$), we shall
get a set $S\neq \emptyset$ closed under this operation, i.e. such that when applying that operation to any $z\in S$
the set neither gets restarted nor gets increased (that is, all the elements computed are already present in $S$). At
this time, Theorem~\ref{T:twisted_EM} implies that the canonical length of the elements in $S$ (which is constant) is
the smallest possible among all positive palindromic-free braids which are $\varepsilon$-twisted conjugated to $x$.
That is, $S\subseteq MPF(x)$.

Now, let $u\in MPF(x)$. Choosing an arbitrary $v\in S$, Corollary~\ref{bottom} tells us that $u$ and $v$ are connected
by a chain of positive, palindromic-free braids of minimal canonical length, each simply $\varepsilon$-twisted
conjugated to the following one. Hence, by construction of $S$, we have $u\in S$. Therefore, $S=MPF(x)$.
\end{proof}

\begin{theorem}\label{tcp}
The twisted conjugacy problem is solvable in the braid group $B_n$.
\end{theorem}

\begin{proof}
Suppose we are given an automorphism $\varphi \colon B_n \to B_n$ (by the images of the generators), and two braids
$u,v\in B_n$. We have to decide whether there exists $x\in B_n$ such that $v=(\varphi (x))^{-1}ux$, and in the positive
case compute such an $x$.

By Theorem~\ref{out}, either $\varphi$ is a conjugation ($\varphi =\gamma_{w}$ for some $w\in B_n$), or it is
$\varepsilon$ followed by a conjugation ($\varphi =\gamma_{w}\varepsilon$ for some $w\in B_n$). We can clearly make
this decision effective, and compute such a $w$. Indeed, in order to check whether $\varphi=\gamma_w$, we need to find
some braid $w$ such that $w^{-1}\sigma_i w=\varphi(\sigma_i)$ for $i=1,\ldots,n-1$. This is an instance of the
so-called {\it multiple simultaneous conjugacy problem} in $B_n$, and algorithms to solve it (and to find such $w$) can
be found in~\cite{Lee_Lee,GM}. On the other hand, checking whether $\varphi=\gamma_{w}\varepsilon$ and finding such $w$
reduces to solving another instance of the multiple simultaneous conjugacy problem in $B_n$: namely, it amounts to find
$w$ such that $w^{-1}\sigma_i^{-1} w=\varphi(\sigma_i)$ for $i=1,\ldots,n-1$. (Alternatively, in our specific
situation, we can make the following conceptually much easier brute force algorithm: knowing, by Theorem~\ref{out},
that there exists $w\in B_n$ such that either $w^{-1}\sigma_i w=\varphi(\sigma_i)$ for $i=1,\ldots,n-1$, or
$w^{-1}\sigma_i^{-1} w=\varphi(\sigma_i)$ for $i=1,\ldots,n-1$, one can always enumerate all words $w\in B_n$ and keep
checking both conditions until finding the good one with the correct $w$.) We can therefore assume that $w$ is known,
and that $\varphi$ is equal either to $\gamma_w$ or to $\gamma_w \varepsilon$.

In the first case $\varphi(x)=w^{-1}xw$, and the equation $v=(\varphi (x))^{-1}ux$ is equivalent to $wv=x^{-1}(wu)x$.
Deciding the existence of such an $x$ and finding it, is just an instance of the standard conjugacy problem in $B_n$
(applied to $wv$ and $wu$), which is well-known to be solvable, see Theorem~\ref{cpbn}.

In the second case, $\varphi(x)=w^{-1}\varepsilon(x)w$, and the equation $v=(\varphi (x))^{-1}ux$ is equivalent to
$wv=(\varepsilon(x))^{-1}(wu)x =\overleftarrow{x}(wu)x$. Deciding the existence of such an $x$ and finding it, is an
instance of the $\varepsilon$-twisted conjugacy problem in $B_n$ (applied to $wv$ and $wu$), which can be solved by
computing the sets $MPF(wu)$ and $MPF(wv)$ (see Corollary~\ref{computeMPF}) and checking whether they coincide or not
(meaning that $wu$ and $wv$ are or are not $\varepsilon$-twisted conjugated, respectively). Notice that, during the
computations of $MPF(wu)$ and $MPF(wv)$, we can keep track of a $\varepsilon$-twisted conjugating element at each step,
so that we can explicitly find a value for $x$ in the case it exists.

We remark that the full computation of the sets $MPF(wu)$ and $MPF(wv)$ will usually not be necessary. We can start the
construction of both sets simultaneously, and kill the whole process giving a positive answer, as soon as we find an
element $z$ in common in both sets (since, in this case, both $wu$ and $wv$ are $\varepsilon$-twisted conjugated to
$z$, an so to each other).
\end{proof}

\section{The conjugacy problem for some extensions of $B_n$.}

\begin{theorem}\label{od}
Every finitely generated subgroup $A\leqslant Aut(B_n)$ is orbit decidable.
\end{theorem}

\begin{proof}
Let $\varphi_1,\ldots ,\varphi_m \in Aut(B_n)$ be given, and consider $A=\langle \varphi_1,\ldots ,\varphi_m\rangle
\leqslant Aut(B_n)$. For every $i=1,\ldots ,m$, compute $w_i \in B_n$ and $\epsilon_i =0,1$ such that $\varphi_i
=\gamma_{w_i}\varepsilon_i^{\epsilon_i}$ (see the first part of the proof of Theorem~\ref{tcp}).

Given two braids $u,v\in B_n$ we have to decide whether or not $v$ is conjugated to $\alpha(u)$ for some $\alpha \in
A$. If $\epsilon_i=0$ for every $i$, then $A\leqslant Inn(B_n)$ and so, the set $\{ \alpha(u) \mid \alpha \in A \}$ is
a certain collection of conjugates of $u$. In this case, our problem just consists on deciding whether or not $v$ is
conjugated to $u$. This is doable by Theorem~\ref{cpbn}.

Otherwise, the set $\{ \alpha(u) \mid \alpha \in A \}$ is a certain collection of conjugates of $u$ and of
$\varepsilon(u)$. In this case, our problem just consists on deciding whether or not $v$ is conjugated to either $u$ or
$\varepsilon(u)$. This is again doable by two applications of Theorem~\ref{cpbn}.
\end{proof}

The following theorem (and the interesting particular case expressed in the corollary below) are immediate consequences
of Theorems~\ref{ses}, \ref{tcp}, and \ref{od}.

\begin{theorem}\label{cp}
Let $G=B_n \rtimes H$ be an extension of the braid group $B_n$ by a finitely generated group $H$ satisfying conditions
(ii) and (iii) above (for instance, take $H$ torsion-free hyperbolic). Then, $G$ has solvable conjugacy problem. $\Box$
\end{theorem}

\begin{corollary}
For any $\varphi_1,\ldots ,\varphi_m \in Aut(B_n)$, the group
 $$
\langle B_n, t_1,\ldots ,t_m \mid t_i^{-1}\sigma t_i =\varphi_i(\sigma ) \quad (\sigma \in B_n) \rangle
 $$
has solvable conjugacy problem. $\Box$
\end{corollary}

\section*{Acknowledgements}

The authors are grateful to the Centre de Recerca Matem\`atica (CRM-Barcelona), since this research was initiated in
the excellent research atmosphere during participation of both authors in one of the CRM workshops. The first named
author acknowledges partial support from the MEC (Spain), Junta de Andaluc\'{\i}a and FEDER, through projects
MTM2007-66929, MTM2010-19355 and P09-FQM-5112. The second author gratefully acknowledges partial support from the MEC
(Spain) and the EFRD (EC) through project number MTM2008-01550.

\end{document}